\documentclass[12pt]{amsart}
\usepackage{cancel}
\usepackage{amsmath,amscd}

\usepackage[utf8]{inputenc}
\usepackage{amsmath,amsfonts}
\usepackage{amsmath,amsfonts,latexsym}
\usepackage{amscd,amssymb,amsmath}
\usepackage{amsfonts}
\usepackage{amsbsy}
\newcounter{tableonly}
\usepackage{subfigure,tikz}
\usepackage[normalem]{ulem}
\usepackage{color}
\usepackage{xcolor}
\usepackage{enumerate}

\usepackage{mathtools}
\usepackage{microtype}

\usepackage{hyperref}

\newtheorem{lema}{Lemma}[section]
\newtheorem{teo}[lema]{Theorem}

\newtheorem{rema}[lema]{Remark}
\newtheorem{coro}[lema]{Corollary}
\newtheorem{defi}[lema]{Definition}

\newtheorem{algorithm}[lema]{Algorithm}
\newtheorem{exem}[lema]{Example}

\newtheorem{prop}[lema]{Proposition}

\newcommand{\bea}{\begin{eqnarray*}}
\newcommand{\eea}{\end{eqnarray*}}
\newcommand{\zz}[1]{}

\newtheorem{theorem}[lema]{Theorem}

\newtheorem{corollary}[lema]{Corollary}
\newtheorem{lemma}[lema]{Lemma}

\newcommand{\hbx}{\hfill$\Box$}

\newcommand{\bz}{\mathbb{Z}}

\newcommand{\br}{\mathbb R}
\newcommand{\bh}{\mathbb H}
\newcommand{\bc}{\mathbb C}

\newcommand{\bn}{\mathbb N}

\newcommand{\bq}{\mathbb Q}

\newcommand{\g}{\mathrm g}

\newcommand{\Mod}{\mathrm {Mod}}
 \newcommand{\NN}{{\mathbb{N}}}
 \newcommand{\ZZ}{{\mathbb{Z}}}

 \newcommand{\QQ}{{\mathbb{Q}}}

\def\P{\mathcal{P}}

\newcommand{\reg}{{\rm reg}}

\newcommand{\per}{{\rm Per}}

\newcommand{\cA}{{\mathcal A}}

\newcommand{\cAP}{{\mathcal AP}}
\newcommand{\cU}{{\mathcal U}}
\newcommand{\cD}{{\mathcal D}}

\newcommand{\MPer}{{\rm MPer}}

\begin{document}
\title[Dold coefficients of  surface homeomorphisms]{Dold coefficients of quasi-unipotent  homeomorphisms of orientable surfaces}

\author[G. Graff]{Grzegorz Graff$^*$}
\thanks{Research supported by the National Science Centre, Poland,
within the grant Sheng 1 UMO-2018/30/Q/ST1/00228. The third author was also supported by the Slovenian Research and Innovation Agency program P1-0292 and grant J1-4001.}
\address{$^*$ Faculty of Applied Physics and Mathematics, Gda\'nsk University of
Technology, Narutowicza 11/12, 80-233 Gda\'nsk, Poland}

\email{grzegorz.graff@pg.edu.pl, ORCID: \href{https://orcid.org/0000-0001-5670-5729}{0000-0001-5670-5729} }
\author[W. Marzantowicz]{Wac{\l}aw Marzantowicz$^{**}$}
\address{$^{**}$ \;Faculty of Mathematics and Computer Science, Adam Mickiewicz University of
Pozna{\'n}, ul. Uniwersytetu Pozna{\'n}skiego 4, 61-614 Pozna{\'n}, Poland.}
\email{marzan@amu.edu.pl, ORCID: \href{https://orcid.org/0000-0001-5933-9955}{0000-0001-5933-9955}}

\author[{\L}. P. Michalak]{{\L}ukasz Patryk Michalak$^{**}$\,'\,$^{***}$}
\address{$^{***}$ \;Institute of Mathematics, Physics and Mechanics, Jadranska 19, SI-1000 Ljubljana, Slovenia}
\email{lukasz.michalak@amu.edu.pl, ORCID: \href{https://orcid.org/0009-0002-4821-3809}{0009-0002-4821-3809}}


\begin{abstract}

The sequence of Dold coefficients $(a_n(f))$ of a self-map $f\colon X \to X$ forms a dual sequence to the sequence of Lefschetz numbers $(L(f^n))$ of iterations of $f$ under the M{\"o}bius inversion formula. The set $\cAP(f) = \{ n \colon a_n(f) \neq 0 \}$ is called the set of algebraic periods of~$f$. Both the set of algebraic periods and sequence of Dold coefficients play an important role in dynamical systems and periodic point theory.

In this work we provide a description of surface homeomorphisms with bounded $(L(f^n))$ (quasi-unipotent maps) in terms of   Dold coefficients.  We also discuss the problem of minimization of the genus of a surface for which one can realize a given set of natural numbers as the set of algebraic periods. Finally, we compute and list  all possible Dold coefficients and algebraic periods for a given orientable surface with small genus and give some geometrical applications of the obtained results.

\end{abstract}
\maketitle

\section{Introduction} \label{sec:Introduction}

This work has been inspired by the problem posed by J. Llibre and V.\,F. Sirvent in \cite{LlSi3, LlSi4}.
The statement of the problem was the following: consider $\Sigma_\g$ an (orientable or non-orientable) closed surface of genus $\g$ and an arbitrary set $\cA \subset \bn$. Is it true that we can always find
a Morse--Smale diffeomorphism $f\colon  \Sigma_\g \to \Sigma_\g$  such that
the set of algebraic periods of $f$ is equal to $\cA$?
We gave recently a positive answer to this question
in \cite{GMMM}. However, a lot of related problems remained still unsolved, in particular in \cite{GMMM} we did not determine the minimal genus among surfaces realizing the set $\cA$.

   In this paper we study more general situation, focusing on a construction  of  preserving orientation homeomorphisms with given sequence of finitely many nonzero Dold coefficients $a_k(f)$ and algebraic periods  $\cAP(f)$ defined below. The necessary condition for such realization  requires $f$ to be quasi-unipotent i.e. the map induced by $f$  on homology has only roots of unity in its spectrum. This class of maps contains also  Morse--Smale diffeomorphisms (cf. \cite{Shub}), but it does not
  have to meet geometric restrictions in terms of Nielsen--Thurston classification provided by da Rocha \cite{Rocha} and described in \cite{GMMM} in this context.
The object of our interest in this paper are the  quantities considered by A. Dold (\cite{Dold}):

 \begin{equation}\label{equ:algebraic periods}
 a_n(f) = \frac{1}{n}\, \sum_{k|n} \mu(n/k) \, L(f^k) \ \ \text{ and } \ \ \cAP(f) =\{ n \in \bn: \, a_n(f) \neq 0\},
 \end{equation}

 where  $L(f^k)$ is the Lefschetz number of $k$th iteration of $f$ and $\mu: \bn \to \{-1,0,1\}$ is the M{\"o}bius function.

  We will denote the sequence of Dold coefficients $(a_k(f))$ by $\cD (f)$ with the set of algebraic periods  $\cAP(f)$  as its support.

 The application of Dold coefficients   in periodic point theory  was first considered in 90s of XX century \cite{Bab-Bog, Llibre} (in the latter paper under the name of {\it periodic Lefschetz numbers}). They constitute a sequence of integers due to so-called Dold congruences \cite{Dold} (cf. also \cite{BGW}) and are coefficients of the periodic expansion of the sequence $(L(f^n))$ (see Section \ref{periodic expansion}). What is more, they are homotopy invariants. Evidently if $\cD(f) \neq \cD(f^\prime)$, then $f$ is not homotopic to $f^\prime$, thus  $\cD (f)$  can be used for identifying different homotopy classes.
 Also, the algebraic periods and Dold coefficients give
an important insights into the structure of periodic points by cross-referencing the information they provide with the local characteristics of periodic points, which are described through fixed point indices at the respective orbits (see Section~\ref{transversal}).
The realization of this line of research was successfully performed for various classes of spaces and their self-maps:  \cite{Duan, Gr, Guirao-Llibre-Qualit,LPR}.

Let us mention here that the related notion to algebraic periods is $\MPer_L(f)$, the so-called {\it minimal set of Lefschetz  periods} considered for transversal maps and expressed in terms of Lefschetz zeta function.  As it was shown in  \cite{GLM} (and \cite{GMMM} in another, purely algebraic way)  the minimal set of Lefschetz  periods coincides with odd algebraic periods: $\MPer_L(f)= \cAP(f) \cap (2\NN -1)$.

The characterization of $\MPer_L(f)$ and its topological consequences was provided for many spaces, such as  $n$-dimensional torus $T^n$ \cite{Guirao-Llibre-T^n}, the sphere $S^n$ \cite{Guirao-Llibre-S2},
 orientable compact surfaces \cite{LlSi1}, non-orientable compact surfaces \cite{LlSi4}, 
 products of $l$-dimensional spheres \cite{Berrizbeitia}, products of spheres of different dimensions \cite{Sir}, spheres and projective spaces \cite{Cufi}.


 The paper is organized in the following way.

Section \ref{Dold-coef} is devoted to a realization of Dold coefficients.
For an orientation-preserving homeomorphism $f\colon S_\g \to S_\g$ the sequence $(a_n(f))$ is determined by the linear symplectic map $H_1(f)\in Sp(2\g, \bz)$. On the other hand, it is known that every $A\in Sp(2\g, \bz)$  is represented as $H_1(f)$ for some $f$.
Thus, our problem of constructing a quasi-unipotent homeomorphism with a given sequence of Dold coefficients reduces to algebraic problem of finding the appropriate quasi-unipotent matrix.
To this end we first adapt a general theorem of the construction of a symplectic matrix corresponding to a palindromic polynomial (Theorem \ref{roots} and Appendix 1) obtaining  realizations of every admissible spectra of symplectic  quasi-unipotent matrices.

Then we provide a relation joining  the multiplicities of primitive  roots of unity in the spectrum with Dold coefficients and genus  (Theorem \ref{theorem:correspondence_Dold_coeff_and_roots_of_unity_and_genus}). In particular the formula (\ref{genus-ak}) enables us to realize any sequence with finitely many nonzero terms as Dold coefficients of some homeomorphism $f$.

In Section \ref{genus-minimal} for the given set $\cA$ of natural numbers we search for an orientable surface of the minimal genus $g$ such that $\cA$ may be realized as the set of algebraic periods of some homeomorphism of $S_\g$.

At the end of algebraic considerations, for the surfaces of small genus we determine the lists of all possible  sets of algebraic periods and Dold coefficients (Section \ref{list}).

In Section \ref{transversal} we provide a geometrical application of our results to the class of transversal maps, proving that the non-vanishing $a_n(f)$ implies the existence of points of minimal period equal to $n$  if $n$ is odd or of minimal period $k\in \{n,\,\frac{n}{2}\}$ if $n$ is even.

In the final Section \ref{discussion} we summarize the obtained results and sketch some directions of the future research.


\section{Periodic expansion}\label{periodic expansion}

Any map $\psi \colon \bn \to \bc$ is called  an {\it arithmetic function}. Formally this is equivalent to the sequence $(a_n)= (\psi(n))$, but it  is more convenient for some considerations thus commonly used in number theory and  often in the referred papers.

Below we include a short description of the notion of periodic expansion of an arithmetic  function $\psi\colon \bn\to  \bc$.
The notion of periodic expansion  was introduced in  \cite{Mar-NP} (see also \cite{Jez-Mar}), and widely applied in \cite{GLM}, \cite{GLN-P}. This concept seems to provide a natural  approach for studying integral invariants of periodic points. In this language any arithmetic function $\psi$, or equivalently a complex valued sequence $(\psi(n))$,  is represented as the series
\begin{equation}\label{formula:k-periodic expansion}
 \psi(n)= \sum_{k=1}^\infty a_k(\psi) \, \reg_k(n)
\end{equation}
of elementary periodic functions
$$
\reg_k(n)= \begin{cases} 0 \; {\text{if}}\; k\nmid n, \cr    k \; {\text{if}}\; k\mid n, \end{cases}
$$
where  the coefficients $a_k(\psi)\in \bc $ are given by the formula
\begin{equation}\label{formula:k-dold_coefficients}
a_k(\psi) = \frac{1}{k}\sum_{d|k} \mu\left(\frac{k}{d}\right)\psi(d).
\end{equation}

In essence, the periodic expansion serves as a discrete analog of the Fourier expansion. Algebraically, it is connected with Ramanujan sums, with its coefficients being instrumental in examining the behavior of the number sequences associated with periodic points (see \cite{BGW}). Here, it simply provides a convenient framework for analyzing the sequence under study.

The notation $\reg_k$ comes from the fact that $\reg_k$ is the character of regular representation of $\bz_k= \bz/ k \bz$ composed with the canonical projection of $\bz$ onto $\bz_k$.
We say that the periodic expansion of $\psi$  is finite if $a_k=0$ for $k> k_0$. Note that if the  periodic expansion is finite then the sequence $(\psi(n))$ given by (\ref{formula:k-periodic expansion}) is periodic with the period equal to ${\rm LCM}(\{k: a_k(\psi)\neq 0\})$ (cf. \cite{Jez-Mar}).

We say that a sequence $(\psi(n)))$ satisfies Dold congruences if for all $n\in \bn$
\begin{equation}\label{equ:Dold congruences}
 \sum_{k|n} \mu\left(\frac{n}{k}\right) \, \psi(k) \,\equiv 0\, \ \ ({\rm mod}\, n),
  \end{equation}

It follows directly from the definition of $\reg_k(n)$ and the M{\"o}bius inversion formula that every arithmetic function $\psi$ can be written in the form (\ref{formula:k-periodic expansion}), where $a_k(\psi)$ are given by (\ref{formula:k-dold_coefficients})  (cf. \cite[Proposition 3.2.7]{Jez-Mar}). Moreover, $\psi$ is integral valued and satisfies  Dold congruences (\ref{equ:Dold congruences}) if and only if  $a_n(\psi) \in \bz$  for every $n \in \bn $ (and then we call $a_n(\psi)$ Dold coefficients).

\zz{Indeed, denote by $H_{od}(X;\bq)= {\underset{i\cong 1 \mod 2}{\,\oplus\,}} \, H_i(X;\bq)$,  analogously  $H_{ev}(X;\bq)= {\underset{i\cong 0 \mod 2}{\,\oplus\,}} \, H_i(X;\bq)$ the subspaces of $H_*(X;\bq)$, and consequently $H_{od}(f)$ and $ H_{ev}(f)$  the linear maps induced by $f$ on these  subspaces. Then (cf. \cite[(3.1.27)]{Jez-Mar})}

\zz{We have the following characterization of this phenomena (cf. \cite{Bab-Bog} and \cite[Thm 3.1.46]{Jez-Mar}).
\begin{teo}\label{thm:bounded sequences expansion}
Let $L = \{l_n\}$ be a sequence of integers. The following
conditions are equivalent:
\begin{itemize}
\item[i)]{The sequence $a(L)$ of $k$-periodic coefficients of $L$ is integral, i.e. $L$ satisfies
the Dold condition, and $\{l_n\}$ is bounded.}
\item[ii)]{$a(L)$ is integral and there exists a natural number $k_0$ such that $a_k = 0$
for $ k \geq  k_0$.}
\item[iii)]{There exists a natural number $l$ such that $l_n = l_{\rm{GCD}(n,l)}$ and $a_k$ is
integral for all $k\mid l$.}
\item[iv)]{There exists a natural number $l$ and integers $a_d$ for all $d\mid l$  such that
$$l_n= {\underset{d\mid l}\sum}\, a_d\,{\underset{i=1}{\overset{d} \sum}}\, \varepsilon_d^{in}\,,\;\;\; n = 1, 2, \, \dots\, ,$$
where $(\varepsilon_d)^d =1$ is a primitive root of unity. }
\zz{\item[v)]{$\zeta(L; z)$ is a rational function with no zeros or poles inside the unit disc.}
\item[vi)]{$S(L; z)$  is  a  rational  function  which  is  equal  to  zero  at  inﬁnity  and
and regular on the unit disc and which has only simple poles with residues
which are integers.}}
\end{itemize}
\end{teo} }

\section{Realization of Dold coeffiecients on orientable surfaces}\label{Dold-coef}


\subsection{Realization of roots of unity as eigenvalues}\label{symplectic_basis}

We start with recalling some basic facts related to homology structure of orientable surfaces. Let us denote by $S_\g$ a closed orientable surface of genus $\g$.

 The standard generators of the first homology group $H_1(S_\g;\mathbb{Z}) \cong \mathbb{Z}^{2\g} = \left<a_1,\ldots,a_\g,b_1,\ldots,b_\g\right>$ form a symplectic basis for the intersection form
	whose matrix in this basis is
		$$ \Omega = \begin{bmatrix}
	0 & I_\g \cr
	-I_\g & 0
	\end{bmatrix}.
	$$
	If $f\colon S_\g \to S_\g$ is an orientation-preserving homeomorphism, it induces the linear map $H_1(f)$ that preserves the intersection form. Thus, if $A$ is the matrix of $H_1(f)$ in the base $\left(a_1,\ldots,a_\g,b_1,\ldots,b_\g\right)$, then
	$$
	\Omega = A^T \Omega A,
	$$
	so $A$ is a symplectic matrix. The group of all symplectic $2\g \times 2\g$ matrices over $\ZZ$ we denote by $Sp(2\g,\ZZ)$.

\begin{defi}\label{quasi}
A rational linear transformation will be called quasi-unipotent if
all their eigenvalues are roots of unity. We will call a continuous
self-map $f$ of a manifold $M$  {\it quasi-unipotent} if the map $H_k(f)$ induced on $k$th homology is quasi-unipotent for every $k$. 
\end{defi}

Let us mention here a well-known fact, that we will use in the forthcoming part of the paper (see e.g. \cite[Fact 2.5]{GMMM}).
\begin{lemma}
\label{fact:bounded_finite_unipotent} For a homeomorphism  $ f \colon \Sigma \to \Sigma$  of a closed surface $\Sigma$ the following three statements are equivalent:
\begin{itemize}
    \item[(1)] $(L(f^n))$  is bounded,
    \item[(2)] $\cAP(f)$ is finite,
    \item[(3)] $H_1(f)\colon H_1(\Sigma; \bz) \to H_1(\Sigma;\bz )$ is quasi-unipotent.\qed
\end{itemize}
\end{lemma}

From Lemma \ref{fact:bounded_finite_unipotent} it follows that for a homeomorphism $f\colon S_\g \to S_\g$ the set $\cAP(f)$ is finite, equivalently $\mathcal{D}(f)$ has only finitely many nonzero terms, if and only if the spectrum $\sigma(A)$ of the symplectic $2\g \times 2\g$ matrix $A=H_1(f)$ consists only of roots of unity. Furthermore,  we will show that such the spectrum $\sigma(A)$ 
determines uniquely  Dold coefficients $\mathcal{D}(f) = (a_k(f))$.

First of all, we will combine several facts to show that the only obstacle to realize a multiset of primitive roots of unity as the spectrum of some quasi-unipotent map is just parity of the numbers of $1$s, roots of degree $1$, and $-1$s, roots of degree $2$.
Recall that a complex number $\xi_k$ that is a root of unity of degree $k$ is called \emph{primitive} if it is not a root of unity of a smaller degree, i.e. $\xi_k^k=1$, but $\xi_k^d \neq 1$ for every $1 \leq d < k$. Any integer polynomial with $\xi_k$ as a root is divisible by its minimal polynomial called \emph{$k$th cyclotomic polynomial}, denoted by $\varphi_k$. In particular, all primitive roots of unity of degree $k$ are all roots of $\varphi_k$. Their number, equal to the degree of $\varphi_k$, is described by the Euler totient function  $\varphi(n)= \#\{k\leq n \colon\, {\rm LCM}(k,n)=1\}$. Since the spectrum of a quasi-unipotent matrix consists of all roots of its characteristic polynomial, if it contains a primitive root of degree $k$, it includes all $\varphi(k)$ primitive roots of unity of degree $k$.

By $\Mod (S_\g)$ we denote the mapping class group of $S_\g$, i.e. the group of isotopy classes of homeomorphisms of $S_\g$. We have the action of $\Mod (S_\g)$ on $H_1(S_\g, \mathbb Z)$ which induces a homomorphism $\Psi \colon \Mod (S_\g) \to  Sp(2\g,\bz)$. 


 \begin{teo}\label{roots}
   Let $R$ be a multiset consisting of $r_k$ times all primitive roots of unity of degree~$k$, where $\sum_k r_k \varphi(k)=2\g$. Then there exists a quasi-unipotent orientation-preserving homeomorphism $f \colon S_\g \to S_\g$ such that the spectrum of $H_1(f)$ is equal to $R$ if and only if $r_1$ and $r_2$ are even.
      \end{teo}

      \begin{proof}
          The proof is a consequence of a few facts. The first of them is  the classical theorem of  Burkhardt (1896), cf. \cite[Proposition 7.3]{Farb-Marg}.
          It states that the homomorphism  $\Psi \colon \Mod (S_\g) \to  Sp(2\g,\bz)$ is surjective. Hence it provides the geometrical realization of every symplectic matrix $A \in Sp(2\g,\bz)$ as $A= H_1(f)$ 
          by a preserving orientation homeomorphism $f\colon S_\g \to S_\g$. 

          The next one  is a result shown by  Q. Yang, (\cite[Theorem 1]{Yang2}). It states, for a polynomial  $\omega(x) \in \mathbb{Z}[x]$ of degree $2\g$, that there exists a symplectic matrix $A \in {\rm Sp}(2g,\mathbb{Z})$ whose characteristic polynomial $\chi_A(x) = \omega(x)$ if and only if $\omega$ is monic and palindromic. (In  the Appendix 1 we recapitulate the proof of this theorem adapting it to the studied problem and providing an example).
          Therefore $R$ is the spectrum of $H_1(f)$ for some $f$ if and only if there exists a monic palindromic polynomial with $R$ as the multiset of its roots. It remains to observe that all cyclotomic polynomials are palindromic except the first two: $\varphi_1(x) = x-1$ and $\varphi_2(x)=x+1$, and that an even number of $\varphi_1$ and $\varphi_2$ in the product of cyclotomic polynomials is necessary and sufficient condition for it to be palindromic.
          \end{proof}


 \subsection{Ramanujan's sum}
	
	Let
	\begin{equation}\label{ck}
	c_k(n) = \sum_{\substack{1 \leq d \leq k \\ \gcd(d,k)=1}} e^{2\pi i dn/k}
	\end{equation}
	be the sum of $n$-powers of all primitive roots of unity of degree $k$, called the \emph{Ramanujan's sum}. 
   By definition,  $\reg_k(n)$ can be written as
	$$
	\reg_k(n) = \sum_{d=1}^k e^{2\pi i dn/k} = \sum_{d | k} c_d(n).
	$$
    Thus applying the M{\"o}bius inversion formula we get
	$$
	c_k(n) = \sum_{d|k} \mu\left(\frac{k}{d}\right)\reg_d(n)  = \sum_{d | \gcd(k,n) } \mu\left(\frac{k}{d}\right)d.
	$$



	\subsection{ M{\"o}bius inversion formula for posets}
	
	In this short section we  introduce the generalization of the M{\"o}bius function and inversion formula for arbitrary posets (cf. \cite{Stanley}).
	
	Suppose $(\P,\leq)$ is a locally finite partially ordered set, that is each closed interval
	$$
	[a,b] = \{ c \colon a \leq c \leq b\}
	$$
	is finite, and so also open and half-open bounded intervals are finite. The M{\"o}bius function on $(\P,
	\leq)$ is $\mu \colon \P \times \P \to \ZZ$ defined inductively as follows:
	$$
	\mu(a,b) = \begin{cases}
		\ \ \ \ \ \ 1 \text{ \ \ \ \ \ \ \ \ \ \ \ \ \  if $a = b$,}\\
		- \sum\limits_{c\, \in\, [a,b)} \mu(a,c)  \text{\, \  if $a < b$,} \\
		\ \ \ \ \ \ 0 \text{ \ \ \ \ \ \ \ \ \ \ \ \ \  if $a \not\leq b$.}
	\end{cases}
	$$
	Hence there is the relation $\sum\limits_{c\, \in\, [a,b]} \mu(a,c) = 0$. The classical M{\"o}bius function corresponds to $\mu(a,b) = \mu\left(\frac{b}{a}\right)$ defined on the poset $(\NN,|)$, using the relation of divisibility.
	
	For $(\P,\leq)$ we define the \emph{incidence algebra} of functions $f,g\colon \P \times \P \to \ZZ$ such that $f(a,b) = 0$ if $a \not\leq b$ (i.e. $[a,b] = \varnothing$) with addition and scalar multiplication pointwise and convolution $*$ defined by
	$$
	(f*g)(a,b) = \sum_{c \in [a,b]} f(a,c)g(c,b).
	$$
	One can show that the  M{\"o}bius function on $(\P,\leq)$ has an inverse with respect to the convolution in the incidence algebra given by the
	\emph{zeta function} $\zeta(a,b) = 1$ for $a \leq b$ and $0$ otherwise. Therefore the M{\"o}bius inversion formula also holds for the incidence algebra:
	$$
	g(a,b) = (f*\mu)(a,b) =  \sum_{c \in [a,b]} f(a,c)\mu(c,b)
    $$
    is equivalent to
    $$
    f(a,b) = (g* \zeta)(a,b) = \sum_{c \in [a,b]} g(a,c).
	$$

	\subsection{Correspondence between multisets of roots of unity and Dold coefficients}
	
	Suppose we have a quasi-unipotent orientation-preserving homeomorphism $f$ of an orientable surface with the spectrum of $H_1(f)$  described by the sequence $(r_k)$, i.e. $H_1(f)$  has $r_k$ times all primitive roots of unity of degree $k$ as eigenvalues. Obviously, only finitely many $r_k$ are nonzero. Alternatively we use the description of the spectrum by a multiset $S = \{n,n,\ldots,n \ \text{ ($r_n$ times)} \colon r_n \neq 0\}$. For example, the sequence $(r_k) = (2,0,1,1,0,\ldots)$ and multiset $S = \{1,1,3,4\}$ describe the same spectrum consisting of all primitive roots of unity of degree $1$ (twice), $3$ and $4$.

	 Since
	$$
	\sum_{k} r_k c_k(n) = \sum_k r_k \sum_{d|k} \mu\left(\frac{k}{d}\right)\reg_d(n) =  \sum_{k,d}  \mu(d,k) r_k \reg_d(n),
	$$
	the Lefschetz numbers of iterations of $f$ are given by
	$$
	L(f^n) = 2 - \sum_{k} r_k c_k(n) = \sum_{k,d}   \mu(d,k) m_k \reg_d(n),
	$$
	where $m_k = -r_k$ for $k\geq 2$ and $m_1 = 2 - r_1$.
	
	By the M{\"o}bius inversion formula the sequence $a_n(f)$ is uniquely determined, so from the equality
	$$
	\sum_{d|n} a_d(f)  d  = L(f^n) = \sum_{d|n} \left( \sum_k  \mu(d,k) m_k\right) d
	$$
	we get
    $$
	a_n(f) = \sum_k \mu(n,k) m_k = \sum_{k=1}^{K} \mu(n,k) m_k,
    $$
	where $K = \max\{ k \colon r_k \neq 0\}$. In particular, $a_n(f)=0$ for $n > K$.
	
	Consider the poset $(\NN, \preceq )$ defined by the relation $a  \preceq  b$ if and only if $b | a$. The corresponding  M{\"o}bius function $\widetilde{\mu}$ is given by $\widetilde{\mu}(a,b) := \mu(b,a) = \mu(\frac{a}{b})$. 	We get
	\begin{equation}\label{formula:a_nf-->m_k}
		a_n(f) = \sum_{k \colon n| k} \mu(n,k)m_k = \sum_{ K!\, \preceq \,k\, \preceq\, n} \widetilde{\mu}(k,n) m_k,		
	\end{equation}
	where we consider these sequences as $a_n(f) = a_f(K!,n)$ and $m_k = m(K!,k)$, where the functions $a_f$ and $m$ are set to be $0$ for other arguments. By the M{\"o}bius inversion formula
	\begin{equation}\label{formula:m_k--->a_nf}
		m_k = \sum_{K! \, \preceq \,n\, \preceq\, k} a_n(f) = \sum_{n \colon k |n } a_n(f) = \sum_n a_{kn}(f).
	\end{equation}


\begin{defi}
	A sequence $(a_n)$ with finitely many nonzero elements is called \emph{realizable} in~${\rm Homeo}^+(S_\g)$ if it is equal to $(a_n(f))$ for an orientation-preserving homeomorphism $f\colon S_\g \to S_\g$.
\end{defi}

	\begin{theorem}\label{theorem:correspondence_Dold_coeff_and_roots_of_unity_and_genus}
		The formulas (\ref{formula:a_nf-->m_k}) and (\ref{formula:m_k--->a_nf}) give a 1-1 correspondence between all sequences $(a_n)$ with finitely many nonzero terms and all sequences $(r_k)$ with finitely many nonzero terms. For a sequence $(a_n)$ the corresponding sequence $(r_k)$ is given by
		\begin{equation}\label{r_k_to_a_k}
		r_1 = 2 - \sum_n a_n, \ \ r_2 = -\sum_n {a_{2n}}, \ \ldots,\ r_k = - \sum_n a_{kn}, \ \ldots
	\end{equation}
		The sequence $(a_n)$ is realizable if and only if the corresponding sequence $(r_k)$ is non-negative and $r_1$ and $r_2$ are even. If it is the case, $(a_n)$ is realizable by a quasi-unipotent orientation-preserving homeomorphism $f\colon S_\g \to S_\g$ on a closed orientable surface of genus $g$ satisfying
		\begin{equation}\label{genus-ak}
		2\g = \sum_k r_k \varphi(k) = 2 - \sum_{n} na_n(f).
		\end{equation}
		Moreover, the spectrum of $H_1(f)$  and genus $\g$ are uniquely determined by $(a_n)$.
	\end{theorem}

	\begin{proof}
		By the description of symplectic matrices on $\ZZ^{2\g}$, a set consisting of only roots of unity can be the spectrum of a symplectic matrix if and only if it is described by $(r_k)$ satisfying the above conditions (Theorem \ref{roots}). 
  Thus the realization of such $(r_k)$ and $(a_n)$ is given by a construction of a homeomorphism on $S_\g$ inducing a linear map on $H_1(S_\g)$ with such a matrix in the standard basis  (see Appendix \ref{AppendixYang}).
The first equality for $2\g$ is just the size of the matrix computed from its spectrum. Moreover, by (\ref{formula:m_k--->a_nf}) 
we get:
		$$
			2g = \sum_k r_k \varphi(k) = 2 - \sum_{k,n} a_{kn}\varphi(k) = 2 - \sum_n \sum_{d | n} a_n \varphi(d) = 2- \sum_{n} a_n\sum_{d|n} \varphi(d)  = 2- \sum_{n} na_n.
		$$
	\end{proof}

\begin{rema}
    Note that using (\ref{formula:a_nf-->m_k}) the formulas for $(a_n)$ in terms of $(r_k)$ are the following:
   \begin{equation}\label{a_k_to_r_k}
   a_1(f) = 2 - \sum_k \mu(k) r_k, \ \ a_n(f) = - \sum_{k\colon n|k} \mu\left(\frac{k}{n}\right)r_k \ \ \text{ for $n \geq 2$}.
	\end{equation}
\end{rema}

\begin{exem}
	A sequence $(a_1,0,\ldots)$ is realizable if and only if $2 \geq a_1$ and $a_1$ is even.
\end{exem}

\begin{exem}
	A sequence $(a_1,a_2,0,\ldots)$ corresponds to $r_1 = 2 - a_1 - a_2$ and $r_2 = -a_2$, $r_n =0$ for $n\geq 3$. Therefore it is realizable if and only if $a_2 \leq 0$ is even and $a_1 \leq  2 -a_2$ is also even. Thus
	$$
	\begin{bmatrix}
		a_1\cr
		a_2
	\end{bmatrix}
=
	\begin{bmatrix}
		2\cr
		0
	\end{bmatrix} +
	x \begin{bmatrix}
		1\cr
		-1
	\end{bmatrix} +
	y \begin{bmatrix}
		-1\cr
		0
	\end{bmatrix},
	$$
	for $x,y \in 2\ZZ_{\geq 0}$.
\end{exem}

\begin{exem}
	Consider a sequence $(a_1,a_2,a_3,0,\ldots)$. The correspondence gives
	$$
	r_1 = 2- a_1 - a_2 - a_3, \ \ r_2 = -a_2 \ \ \text{ and } \ \ r_3 = -a_3.
	$$
	Therefore realizable sequences can be described as follows:
	$$
	\begin{bmatrix}
		a_1\cr
		a_2\cr
		a_3
	\end{bmatrix}
=
	\begin{bmatrix}
		2\cr
		0\cr
		0
	\end{bmatrix} +
	x \begin{bmatrix}
		1\cr
		0 \cr
		-1
	\end{bmatrix} +
	y \begin{bmatrix}
		1 \cr
		-1\cr
		0
	\end{bmatrix} +
	z \begin{bmatrix}
		-1 \cr
		0\cr
		0
	\end{bmatrix},
	$$
	for $x, z \in \ZZ_{\geq 0}$, $y\in 2\ZZ_{\geq 0}$ and $x \equiv z  \mod 2$.
\end{exem}

Note that for $(r_1,\ldots,r_K,0,\ldots)$, $a_K = - r_K$ if $K\geq 2$, so a nonzero $a_K$ in this case have to be negative for $(a_n)$ to be realizable.

\begin{corollary}\label{corollary:necessary_conditions_for_realizing_a_n}
	If a sequence $(a_n)$  with finitely many nonzero terms is non-positive and $\sum_n a_n$ and $\sum_{n} a_{2n}$ are even, then it is realizable. \qed
\end{corollary}

On the other hand,  realizable sequences can have some negative entries, for example
the sequence $(a_n) = (1,2,1,-1,0,-1,0,\ldots)$ corresponds to $(r_k) = (0,0,0,1,0,1,0,\ldots)$ ($S = \{4,6\}$).

\begin{rema}
	For a finite poset $(\P, \leq)$ with enumerated elements $p_1,\ldots,p_K$ in the way compatible with the order, i.e. $p_i \leq p_j$ implies $i\leq j$, one can consider each function $f$ in the incidence algebra of $(\P, \leq)$ as an upper triangular matrix $A_f$ defined by $A_f = [f(p_i,p_j)]_{i,j}$. Then the incidence algebra is isomorphic to the subalgebra of all upper triangular matrices of size $K \times K$ with matrix addition and product and scalar multiplication.
	
	Under this isomorphism for $(\{1,\ldots,K\},|) \subset (\NN,|)$, the matrix $A_\mu$ corresponding to the M{\"o}bius function has an inverse which is given by $A_\zeta = [ \zeta(i,j) ]_{ij}$ consisting of only zeros and ones with a nice pattern: $(i,j)$-entry is $1$ if and only if $i \,|\, j$.
	
	If we write the nonzero parts of sequences $(a_n)$ and $(m_k)$ as the column vectors
	$$
	a = [a_1,\ldots,a_{K}]^T,\ \
	m = [m_1,\ldots,m_K]^T,
	$$
	then
	$$
	a = A_\mu m \ \ \text{ and } \ \ m = A_\zeta a.
	$$
	
	The relation $a = A_\mu m$ for odd indices, obtained in a direct way, was used in  PhD thesis of A.~Myszkowski (\cite[Theorem~3.30]{Mysz}) for the realization, by some linear map on $H_1(S_\g)$, of minimal set of Lefschetz periods and the subsequence of $(a_n)$ of odd indices.
\end{rema}


\section{Realization of a given set of numbers as algebraic periods and the problem of minimal genus}\label{genus-minimal}

\subsection{Algebraic periods}\label{genus-minimal1}

$\newline$

Although Dold coefficients determine uniquely the genus of the surface on which they are realized, the same fact does not hold for algebraic periods. In this section we deal with the problem of finding minimal genus  of a realization of a given set of natural numbers $\cA$  as algebraic periods of some orientation-preserving homeomorphism.

First, we present a simple procedure that allows one to find relatively small genus $\g$ in every situation and it is given as an explicit formula (Proposition \ref{prop:realizable_genus_for_AP} below). Next, we describe the problem of finding minimal genus  in the language of integer linear programming (Section~\ref{linear}).

\begin{prop}\label{prop:realizable_genus_for_AP}
	Let $\cA\subset \NN$ be a finite set and take $a_n = -1$ for $n \in \cA$ and $0$ otherwise. We change $(a_n)$ to guarantee that it is realizable.
	Define indices $n_0$ and $n_1$ as follows:
	\begin{enumerate}[1)]
		\item if $\sum_n a_n$ and $\sum_{n} a_{2n}$ are even, then $n_0 := 0 =: n_1$.
		\item if $\sum_n a_n$ and $\sum_{n} a_{2n}$ are odd, then let $n_0$ be the smallest even index in $\cA$. Take $n_1 := 0$.
		\item if $\sum_n a_n$ is odd and $\sum_{n} a_{2n}$ is even, then let $n_1$ be the smallest odd index in $\cA$. Take $n_0 := 0$.
		\item if $\sum_n a_n$ is even and $\sum_{n} a_{2n}$ is odd, then let $n_0$ be the smallest even index in $\cA$, and let $n_1$ be the smallest odd index in $\cA$.
	\end{enumerate}
	Now, change $a_{n_0} := -2$ if $n_0 \neq 0$ and $a_{n_1} := -2$ if $n_1 \neq 0$ which geometrically is an equivalent of adding all roots of unity of degree $n_0$ or $n_1$ to the spectrum of the realization and increases the genus by $n_0$ or $n_1$ (or both). Then
by Theorem \ref{theorem:correspondence_Dold_coeff_and_roots_of_unity_and_genus} the sequence
 $(a_n)$ is realizable by an orientation-preserving homeomorphism $f \colon S_\g \to S_\g$ such that
	\begin{equation}\label{formula:realizable_genus}
	2\g = 2 + \sum_{n \in \cA} n + n_0 + n_1. 
	\end{equation}\qed
\end{prop}

The above described indices $n_0$ and $n_1$ always exists by the parity of considered sums.

\begin{rema}
	The above construction does not provide the minimum genus in general. For example, for $\cA = \{1,2\}$ it gives the sequence $(a_n) = (-2,-2,0,\ldots)$ with $(r_k) = (6,2,0,\ldots)$, so $\g =4$. However, it can be realized on a torus, $\g=1$, for $(r_k) = (0,2,0,\ldots)$ and $(a_n) = (4,-2,0,\ldots)$.
\end{rema}

\begin{exem}
	For $\cA = \{15\}$ the above procedure leads to $a_{15} = -2$ and $a_n = 0$ for $n\geq 15$, since $n_1 = 15$ and $n_0 = 0$ (the case 3) of the above theorem). Then the corresponding spectrum is described by $S = \{1,1,1,1,3,3,5,5,15,15\}$, or equivalently $(r_k) = (4, 0, 2, 0, 2, 0, 0, 0, 0, 0, 0, 0, 0, 0, 2,0,\ldots)$, so $\g = 16$. In fact, one can check using computer program (see the dataset \cite{GMM}), that this is the minimum genus for $\cA = \{15\}$.
\end{exem}

\begin{rema}
	If one can focus on the minimal set of Lefschetz periods $\MPer_L(f) = \cAP(f) \cap (2\NN -1)$, then $\cA= \{15\}$ can be obtained for the surface of genus $14$ for $r_k = 2$ for $k=6,10,30$ and $r_k = 0$ otherwise (so $S = \{6,6,10,10,30,30\}$). Then $\cAP(f) = \{2,15,30\}$ for $a_2 = 2 = a_{15}$, $a_{30} = -2$ and $a_n =0$ otherwise. The genus $14$ is the smallest one for which $\cA = \{15\}$ can be realized as $\MPer_L(f)$ (see the dataset \cite{GMM}).
\end{rema}


\subsection{Determination of all possible Dold coefficients of quasi-unipotent orientation-preserving homeomorphisms of a surface $S_{\g}$ of a given genus $\g$}\label{list}

$\newline$

The formula

\begin{equation}
2\g = \sum_k r_k \varphi(k) = 2 - \sum_{n} na_n(f)
\end{equation}
of Theorem \ref{theorem:correspondence_Dold_coeff_and_roots_of_unity_and_genus} suggests a correspondence between $(a_n(f))$ and partitions $P=(p_1,\ldots,p_
\g)$ of $2\g = \sum_n p_nn$ given by $a_n(f) = -p_n$ for $n \neq 1$ and $p_1 = 2-a_1$. However, Dold coefficients $a_n(f)$ can be both positive or negative, which does not lead to a partition in general. On the other hand, the sequence $(r_k)$ is non-negative, but we need to deal with the Euler function~$\varphi$.

For each $n$ consider the set
$$
\Phi_n := \varphi^{-1}(n) = \{ k \ \colon \ \varphi(k) = n \}.
$$
Then every spectrum of $H_1(f)$ described by $(r_k)$ leads to a partition $P = (p_n)$ of $2g$ given by $p_n = \sum_{k \in \Phi_n} r_k$. Note that the set $\Phi_n$ has an upper bound and so it is computable (see H. Gupta \cite{Gupta}).

Below we provide the following short algorithm for determining all possible Dold coefficients of quasi-unipotent orientation-preserving homeomorphisms of a surface $S_\g$ of a given genus~$\g$:

\begin{algorithm}\label{algorithm}\,
\begin{itemize}
    \item[] INPUT: genus $g \in \NN$.
    \item[] OUTPUT: list of all possible Dold coeffcients of quasi-unipotent orientation-preserving homeomorphisms of $S_g$.
    \begin{itemize}
      \item[1.] For each partition $P=(p_n)$ of $2\g = \sum p_n n$ and each $n$ consider all partitions of $p_n = \sum_{k \in \Phi_n} r_k$ into $|\Phi_n|$ non-negative numbers.
     \item[2.] Compute the corresponding Dold coefficients $(a_k)$ from $(r_k)$ using (\ref{a_k_to_r_k}) and add them to the output.
    \end{itemize}
  \end{itemize}
\end{algorithm}

The correctness of the above algorithm follows directly from Theorem \ref{theorem:correspondence_Dold_coeff_and_roots_of_unity_and_genus}. 

 By means of a computer program, Algorithm \ref{algorithm} allows us to determine all possible Dold coefficients, the sets of algebraic periods and the sets of possible eigenvalues for a given surface with a small genus $\g$.

 We list below the data for a torus (Table \ref{table:Torus_algebraic_period_orient_preserving}) and genus $2$ surface (Table \ref{table:Surface_genus_2_algebraic_period_orient_preserving}).
 We also provide in the public repository of Gda\'nsk University of Technology "Bridge of  Knowledge" the complete lists for other orientable surfaces up to genus $\g=30$ \cite{GMM}.

{\renewcommand{\arraystretch}{1.3}%

\begin{table*}[h]
	\caption{A torus --- all possible spectra $S$, corresponding Dold coefficients $(a_n(f))$ and sets of algebraic periods and minimal sets of Lefschetz periods for a quasi-unipotent orientation-preserving homeomorphism of a torus.}\label{table:Torus_algebraic_period_orient_preserving}
	\centering
	\begin{tabular}{@{}cccc@{}}
		\hline
		$S$ & $(a_n(f))$ & $\cAP(f)$ & $\cAP_{odd}(f)$ \\
		\hline
		\{1, 1\}  & (0, 0) & $\varnothing$ & $\varnothing$ \\
		\{2, 2\}  & (0, 4, -2) & \{1, 2\} & \{1\} \\
		\{3\}  & (0, 3, 0, -1) & \{1, 3\} & \{1, 3\} \\
		\{4\}  & (0, 2, 1, 0, -1) & \{1, 2, 4\} & \{1\} \\
		\{6\}  & (0, 1, 1, 1, 0, 0, -1) & \{1, 2, 3, 6\} & \{1, 3\} \\
		\hline
	\end{tabular}
\end{table*}

\begin{table*}[h]
	\caption{$S_2$ --- all possible spectra $S$, corresponding Dold coefficients $(a_n(f))$ and sets of algebraic periods and minimal sets of Lefschetz periods for a quasi-unipotent orientation-preserving homeomorphism of $S_2$.}\label{table:Surface_genus_2_algebraic_period_orient_preserving}
	\centering
	\begin{tabular}{@{}cccc@{}}
		\hline
		$S$ & $(a_n(f))$ & $\cAP(f)$ & $\cAP_{odd}(f)$ \\
		\hline
		\{1, 1, 1, 1\}  & (-2)  & \{1\}  & \{1\} \\
		\{1, 1, 2, 2\}  & (2, -2)  & \{1, 2\}  & \{1\} \\
		\{1, 1, 3\}  & (1, 0, -1)  & \{1, 3\}  & \{1, 3\} \\
		\{1, 1, 4\}  & (0, 1, 0, -1)  & \{2, 4\}  & $\varnothing$ \\
		\{1, 1, 6\}  & (-1, 1, 1, 0, 0, -1)  & \{1, 2, 3, 6\}  & \{1, 3\} \\
		\{2, 2, 2, 2\}  & (6, -4)  & \{1, 2\}  & \{1\} \\
		\{2, 2, 3\}  & (5, -2, -1)  & \{1, 2, 3\}  & \{1, 3\} \\
		\{2, 2, 4\}  & (4, -1, 0, -1)  & \{1, 2, 4\}  & \{1\} \\
		\{2, 2, 6\}  & (3, -1, 1, 0, 0, -1)  & \{1, 2, 3, 6\}  & \{1, 3\} \\
		\{3, 3\}  & (4, 0, -2)  & \{1, 3\}  & \{1, 3\} \\
		\{3, 4\}  & (3, 1, -1, -1)  & \{1, 2, 3, 4\}  & \{1, 3\} \\
		\{3, 6\}  & (2, 1, 0, 0, 0, -1)  & \{1, 2, 6\}  & \{1\} \\
		\{4, 4\}  & (2, 2, 0, -2)  & \{1, 2, 4\}  & \{1\} \\
		\{4, 6\}  & (1, 2, 1, -1, 0, -1)  & \{1, 2, 3, 4, 6\}  & \{1, 3\} \\
		\{5\}  & (3, 0, 0, 0, -1)  & \{1, 5\}  & \{1, 5\} \\
		\{6, 6\}  & (0, 2, 2, 0, 0, -2)  & \{2, 3, 6\}  & \{3\} \\
		\{8\}  & (2, 0, 0, 1, 0, 0, 0, -1)  & \{1, 4, 8\}  & \{1\} \\
		\{10\}  & (1, 1, 0, 0, 1, 0, 0, 0, 0, -1)  & \{1, 2, 5, 10\}  & \{1, 5\} \\
		\{12\}  & (2, -1, 0, 1, 0, 1, 0, 0, 0, 0, 0, -1)  & \{1, 2, 4, 6, 12\}  & \{1\} \\
		\hline
	\end{tabular}
\end{table*}

\begin{table*}[h]
	\caption{The numbers $\#Sp$, $\#\cAP$ and $\#\cAP_{odd}$ of all possible spectra, sets of algebraic periods and minimal sets of Lefschetz periods, respectively, for a quasi-unipotent orientation-preserving homeomorphism of an orientable surface $S_\g$ of genus $\g$.}\label{table:algebraic_period_orient_preserving}
	\centering
	
	\begin{tabular}{cccc|cccc}
		\hline
		$\g$ & $\#Sp$ & $\#\cAP$ & $\#\cAP_{odd}$ & $\g$ & $\#Sp$  & $\#\cAP$ & $\#\cAP_{odd}$ \\ \hline
		1    & 5      & 5        & 3              & 11   & 39881   & 6629     & 161            \\
		2    & 19     & 15       & 5              & 12   & 76085   & 10939    & 218            \\
		3    & 59     & 40       & 9              & 13   & 141877  & 17469    & 283            \\
		4    & 165    & 93       & 15             & 14   & 259373  & 27672    & 360            \\
		5    & 419    & 192      & 21             & 15   & 465493  & 43011    & 477            \\
		6    & 1001   & 381      & 33             & 16   & 821813  & 66035    & 613            \\
		7    & 2257   & 719      & 49             & 17   & 1428725 & 100032   & 772            \\
		8    & 4877   & 1322     & 64             & 18   & 2449573 & 150154   & 982            \\
		9    & 10133  & 2317     & 88             & 19   & 4145249 & 222234   & 1224           \\
		10   & 20399  & 3977     & 125            & 20   & 6931259 & 326944   & 1522    \\
		\hline    
		  
	\end{tabular}

\end{table*}

 Let us remark, that these lists coincide in respect to odd algebraic periods (i.e. sets of minimal Lefschetz periods $\MPer_L(f)$) to their counterpart  given in  \cite{LlSi1} for small $\g$, as in \cite{GLM} and the database described therein.
 However, we provide also data for even values. In Table \ref{table:algebraic_period_orient_preserving} we may observe that the number and growth rate of admissible values for both odd and even data is much greater than for only odd ones.

   On the other hand, the knowledge of all (odd and even data) allows us to draw much more information about the map behaviour of homeomorphism on a given surface $S_{\g}$ and the structure of their periodic points (see the application to periodic point theory described by Remark \ref{2} below).

\begin{rema}\label{2}
J. Nielsen in \cite{Niel1} proved that for an  orientation-preserving self-homeomorphism $f$ of $S_{\g}$
at least one of $f$, $f^2$, $f^3$, \ldots, $f^{2\g-3}$, $f^{2\g-2}$ has a fixed-point if $\g \geq3$  and at least one of $f$, $f^2$, $f^3$ has a fixed-point if $\g= 2$. The open problem remained if at least one of $f$, $f^2$ has a fixed-point if $\g= 2$. This conjecture was proved in \cite{Di-Ll}. Note  however, that for the class of orientation-preserving quasi-unipotent homeomorphisms the validity of the conjecture results from the fact that for every $f\colon S_2 \to S_2$ in this class its set of algebraic periods ${\cAP}(f)$ includes either $1$ or $2$ (see Table \ref{table:Surface_genus_2_algebraic_period_orient_preserving}) and from Lefschetz fixed point theorem. Our computations up to genus 30 (see \cite{GMM}) confirm the Nielsen's result and show that it cannot be improved on the algebraic level since one can always find a self-homeomorphism of $S_\g$, $3 \leq \g \leq 30$, whose minimal algebraic period is $2\g-2$.
\end{rema}

\subsection{Computing the minimal genus for a given set of algebraic periods}\label{genus-minimal2}

$\newline$

Algorithm \ref{algorithm} can be used in the brute force method for finding the minimum genus $\g$ of a surface $S_\g$ to realize a given set $\cA \subset \NN$ as the set of algebraic periods of an orientation-preserving self-homeomorphism of $S_\g$.

\begin{prop}\label{proc}
    There is a procedure that for any finite subset $\cA$ of natural numbers determines in finitely many steps the minimal number $\g$ together with Dold coefficients $a_n(f)$ of an orientation-preserving homeomorphisms $f\colon S_\g \to S_\g$ such that $\cAP(f) = \cA$.
\end{prop}

\begin{proof}
    The formula (\ref{formula:realizable_genus}) in Proposition \ref{prop:realizable_genus_for_AP} provides an upper bound $C > 0$ for the minimum genus. Now, for each $\g=1,2,\ldots,C$ using Algorithm \ref{algorithm} determine all possible Dold coefficients and sets of algebraic periods and stop, when you first obtain $\cA$.
\end{proof}

In general, the problem of finding the minimum genus $\g$ for realizing a finite subset $\cA \subset \NN$ as the set of algebraic periods of an orientation-preserving homeomorphism of a surface can be formulated in terms of \emph{integer linear programming}.

\textit{
\label{linear}
	For a given set $\cA \subset \NN$ with $K = \max(\cA)$ minimize the function
	$$
	2\g(  a_n \colon n\in \cA) := 2- \sum_{n\in \cA} na_n,
	$$
	for $a_n \in \ZZ \setminus \{0\}$ and $x,y \in \ZZ$, subject to
	$$
	  \sum_{n \in \cA} a_n \leq 2, \ \ \sum_{2n \in \cA}  {a_{2n}} \leq 0, \ \ldots,\  \sum_{kn \in \cA}  a_{kn} \leq 0,  \ \ldots,\  \sum_{Kn \in \cA}  a_{Kn} \leq 0
	$$
	and
	$$
	\sum_{n \in \cA} a_n - 2x = 0, \ \ \sum_{2n \in \cA} a_{2n} - 2y = 0
	$$
	(the last two equations follows from the parity of $r_1$ and $r_2$).}

Note that the function to be minimized is integer-valued and non-negative (it is twice the genus of a surface), so the solution always exists. Using specialized algorithms for solving integer linear programming problem one can minimize a genus in a better way than using the brute force method presented in Proposition \ref{proc}. However, for small genera the brute force method is sufficient. In the database \cite{GMM} we included lists of sets of algebraic periods for a which a given genus $g=1,\ldots,30$ is minimal.

\section{Dynamical application}\label{transversal}

In this section we show the existence of smooth surface self-maps with a given finite sequence of Dold coefficients for which the number of points of minimal period $k$ is estimated from below by $\vert a_k(f)\vert$. It is a consequence of the results of previous sections combined with some classical facts.

We begin with a well-known but useful fact, which states that
in the statements of all facts presented in previous sections a  homeomorphism  can be replaced by a diffeomorphism of a surface.
This follows from \cite[Theorem 1.10]{Farb-Marg} which says that every homeomorphism of a surface is isotopic to a diffeomorphism.

 To formulate consequences of our main result we make use of the notion of transversal maps.
  \begin{defi}[\cite{Dold}]\label{defi:transversal}
Let $f\colon  \cU \to X$
be a~$C^1$-map of an open subset $\cU$ of a
manifold  $X$, $\dim X=d$. We say that
 $f$ is transversal,   if for
any $m \in {\bn}$ and $x \in X$ such that $f^m(x)=x$, 
the condition
$
1 \notin  \sigma(Df^m(x))
$
holds, where $\sigma(Df^m (x))$ denotes the spectrum of the derivative of $f^m$ at $x$.
 \end{defi}
 This notion is slightly more general than the notion of hyperbolic maps (cf. \cite{Franks}).

\zz{ The difference between hyperbolic  transversal is the following.
   In the first case  $\sigma(Df(x))\cap \{\vert z \vert =1\} = \emptyset$. In the second case we require only  $\sigma(Df(x))\cap  \mathcal{C}  = \emptyset$, with
   $\mathcal{C}= {\underset{d\geq 1}\bigcup} \mathcal{C}_m$, where $ \mathcal{C}_m  \subset \{\vert z \vert =1\}$ is the set of roots of unity of degree $m$.
The set of all transversal maps is denoted by  $C^1_T(\cU \to X)$, and the set of all hyperbolic maps $C^1_H( \cU \to X)$ respectively. Obviously $C^1_H( \cU \to X)  \subset C^1_T(\cU \to X)$.}

A well-known  geometric property of a transversal, thus a hyperbolic, map states that
for any map  $f\in C^1_T(X)$,  and every $m\in\bn$ the set $P^m(f)=\{x\in X: f^m(x)=x\}$
consists of isolated points and thus is finite  if $X$ is closed (cf.  \cite[Proposition 5.7]{Dold} or  \cite{Seidel}).

\zz{The above proposition states that each set $P^m(f)$, thus each  $P_m(f)= P^m(f) \setminus {\underset{k|m, k<m}{\,\bigcup\,}} P^k(f) $ is finite. But the entire set of periodic points $P(f)= {\underset{m}{\,\bigcup\,}} P^m(f) = {\underset{m}{\,\bigcup\,}} P_m(f)$ may be infinite.}

 The main property of the  class of transversal maps is
given in the following theorem which is widely used but whose
rigorous proof is complicated (see \cite{Dold} if $X= \br^d$, and \cite{Seidel} for the general case, also \cite{Jez-Mar} for an exposition).

\begin{teo}\label{thm:Seidel}
The set  $C^\infty_T(\cU, X)= C^1_T(\cU,X) \cap C^\infty(\cU,X)$   is generic in $C^0({\cU}, X)$,
i.e.  $
C^\infty_T(\cU, X) =\bigcap_{n=1}^{\infty} G_n$ where each
$G_n$ is open and  dense in $C^0({\cU}, X)$.
 In particular every map $ f\colon \cU\to X$ is homotopic to a~transversal map $h\colon  \cU\to X$.
\end{teo}
Note that the second part of the statement  follows from the
first, because two maps close in  $C^0$-topology   are
homotopic. In particular this theorem applies to self-maps $C^0(X)$ of a smooth  closed manifold $X$.

\zz{Finally, all mappings in $C^1_T(X)$, thus  $C^1_H(X)$, have the following geometric property which states that non-vanishing $a_n(f)$, $n$-odd, implies that $P_n(f)\neq \emptyset$.
 We have}
\zz{ \begin{teo}\label{thm:periodic for transversal} {\phantom{newline}}
\newline   Let $f\colon X\to X$  be a transversal map, e.g. a hyperbolic map

  If $a_n(f) \neq 0\,$ then
$ \;\begin{cases} P_n(f) \cup  P_{\frac{n}{2}}(f) \neq  \emptyset \;  \, {\text{if}}\; n \; {\text{is even}}, \cr
P_n(f) \neq  \emptyset \; \; {\text{if}}\; n\; {\text{is odd}}.
\end{cases}
$
\newline
Moreover
$ \;\begin{cases} \vert P_n(f) \cup  P_{\frac{n}{2}}(f) \vert  \geq \vert a_n(f)\vert  \;  \, {\text{if}}\; n \; {\text{is even}}, \cr
\vert P_n(f) \vert  \geq  \vert a_n(f) \vert   \; \; {\text{if}}\; n\; {\text{is odd}}.
\end{cases}
$
\end{teo}
\begin{proof}

The first part of statement is shown in \cite{Franks} for the hyperbolic maps and in \cite{Dold} for transversal maps
 (see also {\cite[Corollary 3.3.10]{Jez-Mar}} for its presentation).
  The second part of statement follows from formula \cite[3.3.9]{Jez-Mar} of \cite[Proposition 3.3.8]{Jez-Mar}.
 \end{proof}}

\zz{Combining Theorems  \ref{theorem:correspondence_Dold_coeff_and_roots_of_unity_and_genus},  \ref{thm:Seidel}, \ref{thm:periodic for transversal}}

We are in a position to formulate the dynamical (geometrical) consequence of our work. Denote by $P_n(f)= P^n(f)\setminus {\underset{k|n,\, k<n} \bigcup} P^k(f)$ the set of points with the minimal period equal to~$n$, and by ${\rm Per}(f)=\{n\in \mathbb{N}: P_n(f)\neq \emptyset\}$ the set of minimal periods of $f$.

\begin{teo}\label{thm:main geometrical}
Let $\cD = (a_k)$ be a sequence of integers with a finite support $\cA$. Suppose that $\cD$ satisfies realizability conditions of Theorem \ref{theorem:correspondence_Dold_coeff_and_roots_of_unity_and_genus}.
Then there exist an orientable surface $S_\g$ of genus $\g$ given by the formula (\ref{genus-ak}) and a preserving orientation  diffeomorphism $f\colon S_\g \to S_\g\,$, $f\in C^1(S_\g)$,  such that $\cD(f) = \cD$ is the sequence of Dold coefficients of $f$.

Furthermore for every  transversal map  $f_T\in C^1_T(S_\g)$ homotopic to $f$  we have 
$$
\vert P_n(f_T) \cup  P_{\frac{n}{2}}(f_T) \vert  \geq \vert a_n(f)\vert  \;  \, {\text{if}}\; n \; {\text{is even}}\;\;\;\text{and} \;\; \;
     \vert P_n(f_T) \vert  \geq  \vert a_n(f) \vert   \; \; {\text{if}}\; n\; {\text{is odd}}.
$$
In particular $ \cA_{odd} = \cAP_{odd}(f) \subset \per(f_T)$.
\end{teo}
\begin{proof}
The existence of $f$ follows from Theorem \ref{theorem:correspondence_Dold_coeff_and_roots_of_unity_and_genus}. By Theorem \ref{thm:Seidel} we can find a transversal map $f_T$ in the homotopy class of $f$.
The second part of the statement follows from results of  \cite{Franks} for the hyperbolic maps and \sout{in} from those of \cite{Dold} for transversal maps and was also shown independently by  S. N. Chow, J. Mallet-Paret and J. A. Yorke
 (see \cite[Proposition 3.3.8]{Jez-Mar} and  {\cite[Corollary 3.3.10]{Jez-Mar}} for its presentation).
\end{proof}
\zz{The above let us formulate the analytical (dynamical) consequence of Theorem \ref{thm:main1}. Put $\cA_{od}(f) = \cA(f)\cap \{2\bn -\}$, and $\cA_{ev}(f)=\cA(f)\cap \{2\bn\}$.
\begin{coro}\label{coro: main analytical}
Let $\cA \subset \bn$ be a finite subset of natural numbers and $f\colon S_\g\to S_\g$ be a preserving, or reversing, orientation  Morse--Smale diffeomorphism of oriented surface of genus $\g$ given by Theorem \ref{thm:main1}, or correspondingly a Morse--Smale diffeomorphism of non-orientable surface $N_\g$   given by Theorem \ref{thm:main1},  such that $\cA(f)= \cA$.

Then for every $n \in \cA_{od}(f)  $ and every $ h \in C_T(S_\g)$, respectively $h \in C_T(N_\g)$, with  $h \sim f$ we have $n\in \per(h)$.
For every $n \in \cA_{ev}(f)  $ and every $ h \in C_T(S_\g)$ with $h \sim f$,  respectively $h \in C_T(N_\g)$  with $h \sim f$, we have $(\{n\} \cup  \{\frac{n}{2}\} )\cap \per(h) \neq \emptyset $.
\end{coro}}

\section{Final discussion}\label{discussion}

At the end, we would like to display the difference between  this and the previous paper \cite{GMMM} and summarize our approach. Here the geometrical realization of a preserving orientation homeomorphism is a direct consequence of the classical Burkhardt theorem which states that every symplectic matrix $A\in Sp(2\g, \bz)$ can be realized as $H_1(f)$ for some $f$. On the other hand, the homotopy (thus isotopy) classes of  Morse--Smale diffeomorphisms correspond to the classes $T_1$ and $T_2$ of the Nielsen--Thurston classification of homeomorphisms of closed compact surfaces \cite{Farb-Marg}. Consequently, the realization of a given finite set $\cA$ as $\cAP(f)$ of a Morse--Smale diffeomorphism  demands a geometric construction which requires "a space", i.e.  the corresponding genus $\g$ should be big enough  (cf. \cite{GMMM}, see also \cite{Gri-Med-Poch} for the recent developments in the study of Morse--Smale diffeomorphisms f). We must emphasize that it is not sure if the epimorphism $\Psi\colon Mod(S_g) \to Sp(2\g,\bz)$  maps $\{T_1 \cup T_2\}$ onto a subset of symplectic group consisting of all quasi-unipotent matrices. The converse hypothesis that  the counter-image $\Psi^{-1}$  of quasi-unipotent matrices is contained in  $T_1 \cup T_2$ had been conjectured by J. Nielsen in \cite{Niel2} where he showed that the image  $\Psi(T_1\cup T_2)$  is contained in the set of   quasi-unipotent matrices. This conjecture was disproved
by  W. Thurston (see \cite{Farb-Marg}) who gave an example of pseudo-Anosov diffeomorphism $f$ such that $H_1(f)= {\rm id}$, i.e.  belonging to the class $T_4$ of the Nielsen--Thurston classification.  In particular it has infinitely many periodic points. The asymptotic growth rate of periodic points of a pseudo-Anosov diffeomorphism $f$ of a hyperbolic surface is equal to the asymptotic Nielsen number, which is equal to the value  of stretching factor of $f$ (cf. \cite{BoJia}).    This
shows that the knowledge about periodic points encoded in the matrix $H_1(f)$ is rather limited in general.

As a result, the dynamical contribution of our work is to show the existence of a (transversal) self-map $f$ of an orientable surface which has periodic points of periods $n$ if $n$ is odd, and of  at least one of $n$ or $n/2$ if $n$ is even with $n$ belonging to a given $\cA$. Moreover the cardinalities of $P_{n}(f)$ are estimated from below for all transversal maps homotopic to $f$ by $\vert a_{n}\vert $.


 The main \zz{technical} contribution of this work is the algebraic construction which gives  Theorem \ref{roots} and its  specifications Theorem \ref{theorem:correspondence_Dold_coeff_and_roots_of_unity_and_genus} and Proposition \ref{prop:realizable_genus_for_AP}. Additionally, we present the proof of main theorem in a way which has algorithmic character and can lead to an explicit construction of such  a map. 
 As the reader may note, there are three steps in a  procedure of finding such a map.

\begin{itemize}
\item{Algebraic: for a given sequence $\cD = (a_n)$ of Dold coefficients with finite support $\cA$ find a symplectic $2\g \times 2\g$  matrix $A$,  of possible small size  such that for every homeomorphism $f\colon S_\g\to S_g$ with $H_1(f)=A$ we have $\cD(f) =\cD$ and $ \cAP(f)= \cA$. This part is the main subject studied in our work.}
  \item{Geometric: for a given symplectic matrix $A\in Sp(2\g, \bz)$ find a homeomorphism, thus diffeomorphism $f\colon S_\g \to S_\g$ such that $\Psi(f) = A$.  In the original proof of Burkhardt (cf. \cite[Thm 6.4]{Farb-Marg}) for each elementary symplectic matrix ${\bf e}$, is shown a construction of a homeomorphism $F\colon S_\g \to S_\g$, called a Burkhardt generator, such that $\Psi(F)= {\bf e}$. Since elementary matrices form a set of generators of $Sp(2\g, \bz)$, it is enough to represent $A$ as a product of elementary matrices, but it is not easy procedure in general.}
          \item{Analytic: for a given homeomorphism $f\colon S_\g \to S_\g$ find in an effective way a transversal map  $\tilde{f}: S_\g \to S_\g$ homotopic to $f$, e.g. close to $f$.}
  \end{itemize}

  The latter appears to be an interesting  question, the resolution of which could render the described procedure fully effective.

\section{Appendix}\label{AppendixYang}

In  Appendix we describe an effective construction of a quasi-unipotent symplectic matrix with a given characteristic polynomial $f$ of degree $2\g$ which is a product of cyclotomic polynomials with even multiplicites of $\varphi_1(x) = x-1$ and $\varphi_2(x) = x+1$ (cf. Theorem \ref{roots}). The existence of such a matrix is provided by the following theorem.
\begin{theorem}[Q. Yang, {\cite[Theorem 1]{Yang2}}]
		Let $f(x) \in \mathbb{Z}[x]$ be a palindromic monic polynomial of degree $2\g$. Then there exists a symplectic matrix $A \in {\rm Sp}(2\g,\mathbb{Z})$ whose characteristic polynomial $\chi_A = f$. \hbx
	\end{theorem}

	The proof is divided into two steps. First, the theorem is showed if $f$ is an irreducible polynomial \cite[Corollary 1]{Yang1}. Then, the result is easily extended to all palindromic monic  polynomials.

    Note that if we change the order in the standard symplectic basis (cf. Section \ref{symplectic_basis}) to $(a_1,b_1,\ldots,a_\g,b_\g)$, then the matrix of the symplectic form is block diagonal consisting of $\g$ times
    $$ \Omega = \begin{bmatrix}
	0 & 1 \cr
	-1 & 0
	\end{bmatrix}.
	$$
    From this it is easily seen that the direct sum of symplectic matrices is symplectic. Thus it is sufficient to show a construction of symplectic matrix for each cyclotomic polynomial $\varphi_k(x)$, $k\geq 3$ ($(\varphi_1(x))^2$ is realized by $I_2$ and $(\varphi_2(x))^2$ by $-I_2$). The below considerations are based on and are adapted version of \cite{Yang1}.

    Fix a cyclotomic polynomial $f=\varphi_k$, $k\geq 3$, so it is palindromic monic and irreducible of degree $\deg(f) = 2\g$. Let $\xi$ be a root of $f$ and $K = \QQ(\xi)$. The extension $K/\QQ$ is Galois of degree $2\g$, and so ${\rm Gal}(K/\QQ) = \{\sigma_1,\ldots,\sigma_{2\g}\}$ (it is abelian and isomorphic to multiplicative group $(\ZZ/k)^\times$). Assume $\sigma_1 = {\rm id}_K$ and $\sigma_2(\xi) = 1/\xi$, since $1/\xi$ is also a root of $f$ (this is clear for cyclotomic polynomials, but generally, because of palindromicity of $f$). Define the conjugation in $K$ by using $\sigma_2$, i.e. $\widetilde{x} := \sigma_2(x)$ for $x\in K$. Moreover, the notation $x^{(i)} := \sigma_i(x)$ is used.

	\subsection{Construction of a symplectic matrix with a given cyclotomic polynomial as its characteristic polynomial}\label{constr}

	\begin{enumerate}
		\item Consider the basis $\beta = (\beta_1,\ldots,\beta_{2\g})^T := (1,\xi,\ldots,\xi^{2\g-1})^T$ of $\mathbb{Z}[\xi]$ and the conjugated basis $\widetilde{\beta} = (1,\widetilde{\xi},\ldots,\widetilde{\xi}^{2\g-1})^T$ of  $\widetilde{\mathbb{Z}[\xi]} = \ZZ[\widetilde{\xi}]= \mathbb{Z}[\xi]$ (because $1/\xi = (1-f(\xi))/\xi \in \ZZ[\xi]$ since the constant term is $f(0)=1$).
		
		\item Take the dual basis $\beta' = (\beta'_1,\ldots,\beta'_{2\g})^T$ to $\beta$ with respect to the trace
		$$
		Tr(x) = \sum_{i=1}^{2\g} x^{(i)},
		$$
		i.e. $Tr(\beta_i\beta'_j) = \delta_{ij}$. The trace $Tr \colon K \to \QQ$ is a $\QQ$-linear homomorphism such that if $x \in \ZZ[\xi]$, then $Tr(x) \in \ZZ$.
		
		\item It is known that $\ZZ[\xi] = \Delta\ZZ[\xi]'$, where $\Delta := \xi^{1-g}f'(\xi)$ and $f'$ is the derivative of $f$. We take a matrix $M\in {\rm GL}(2\g,\ZZ)$ such that $M\widetilde{\beta} = \Delta \beta'$. It can be shown that $M$ is skew-symmetric, i.e. $M^T = -M$.
		
		\item Because $M$ is unimodular, it has all invariant factors $1$, the same as $\Omega$, and so by \cite[Theorem IV.3]{Newman} these two matrices are congruent, i.e. there is a matrix $Q \in {\rm GL}(2\g,\ZZ)$ such that $M = Q^T \Omega Q$ (in fact, $M$ is non-singular and skew-symmetric, so it defines a symplectic form, and so existence of the matrix $Q$ is provided by the choice of Darboux basis for a symplectic space determined by $M$, and the algorithm for the construction of Darboux basis shows that $Q$ can have integer coefficients).
		
		\item Take $\alpha = Q\beta$. Then $\xi\alpha$ is another basis for $\ZZ[\xi]$, and so there is a matrix $A \in {\rm GL}(2\g,\ZZ)$ such that $A\alpha = \xi \alpha$. Therefore $\xi$ is an eigenvalue of $A$ and because of irreducibility of $f$ and size of $A$, $\chi_A = f$. It is proven in \cite{Yang1} that $A\in {\rm Sp}(2\g,\mathbb{Z})$.
	\end{enumerate}
		
	\begin{exem}
	    Let $f(x) = \varphi_3(x) = x^2 + x +1$ be the 3rd cyclotomic polynomial and $\xi = e^{2\pi i/3} = \frac{-1+i\sqrt{3}}{3}$.
	\end{exem}

      Note that $\widetilde{\xi} = 1/\xi = \xi^2 = - \xi -1$. Moreover, ${\rm Gal}(K/\QQ) = \{\sigma_1,\sigma_{2}\}$.
    \begin{enumerate}
        \item  Take the basis $\beta = (1,\xi)^T$ of $\ZZ[\xi]$ and the conjugated basis $\widetilde{\beta} = (1,-1-\xi)^T$.

      \item The trace $Tr$ is given by
      $$
      Tr(a+b\xi) = a+b\xi + \widetilde{a+b\xi} = a+b\xi + a +b(-1-\xi) = 2a-b.
      $$
	   If we consider a bilinear form $Tr \colon K \times K \to \QQ$ given by $(x,y) \mapsto Tr(x\cdot y)$, then its matrix in the basis $\beta$ is
		$$
        M_{Tr} = \begin{bmatrix}
    	2 & -1 \cr
    	-1 & -1
    	\end{bmatrix}.
	   $$

    Therefore the dual basis $\beta'$ satisfies $(\beta')^T M_{Tr} \beta = I_{2}$, so in the basis $\beta$, the dual vectors $\beta'_i$ are columns of the matrix
		$$
        M_{Tr}^{-1} = \begin{bmatrix}
    	1/3 & -1/3 \cr
    	-1/3 & -2/3
    	\end{bmatrix},
	   $$
    so $\beta' = (\frac{1}{3} - \frac{1}{3}\xi,-\frac{1}{3}-\frac{2}{3}\xi)^T$

 \item Compute $\Delta = \xi^{1-\g}f'(\xi) = 2\xi+1$ and $\Delta \beta' = (1+\xi,1)^T$, since e.g. $(2\xi+1)\cdot (\frac{1}{3} - \frac{1}{3}\xi) = \frac{1}{3}(2\xi+1-2\xi^2-\xi) = \frac{1}{3}(\xi+1+2+2\xi) = 1 + \xi$.

 Now, to find $M$ such that $M\widetilde{\beta} = \Delta\beta'$, write these vectors as matrices in the basis $\beta$, getting
		$$
        M_{\widetilde{\beta}} = \begin{bmatrix}
    	1 & 0 \cr
    	-1 & -1
    	\end{bmatrix}
     \text{\ \  and \ \ }
        M_{\Delta\beta'} = \begin{bmatrix}
    	1 & 1 \cr
    	1 & 0
    	\end{bmatrix},
	   $$
    so
    $$
    M = M_{\Delta\beta'}\cdot M_{\widetilde{\beta}}^{-1} = \begin{bmatrix}
    	0 & -1 \cr
    	1 & 0
    	\end{bmatrix}.
    $$

 \item Since $M^T = -M$, we change the basis to Darboux basis producing a matrix $Q$ such that $M = Q^T\Omega Q$. In this case, we just need to change the order of columns and rows using matrix
 $$
 Q = \begin{bmatrix}
    	0 & 1 \cr
    	1 & 0
    	\end{bmatrix}.
 $$

    \item Now, take a new basis $\alpha = Q\cdot\beta = (\xi,1)^T$ and consider $\xi\alpha = (-1-\xi,\xi)^T$. Again, we write the bases as matrices in the basis $\beta$:
        $$
        M_{\alpha} = \begin{bmatrix}
    	0 & 1 \cr
    	1 & 0
    	\end{bmatrix}
     \text{\ \  and \ \ }
        M_{\xi\alpha} = \begin{bmatrix}
    	-1 & -1 \cr
    	0 & 1
    	\end{bmatrix},
	   $$
    so
    $$
    A = M_{\xi\alpha}\cdot M_{\alpha}^{-1} = \begin{bmatrix}
    	-1 & -1 \cr
    	1 & 0
    	\end{bmatrix}
    $$
    satisfies $A\alpha = \xi \alpha$. Indeed, its characteristic polynomial is
    $$
    \chi_A(x) = \det(xI_2 -A ) = \det \begin{bmatrix}
    	x+1 & 1 \cr
    	-1 & x
    	\end{bmatrix}
     = x^2 +x +1 = \varphi_3(x).
    $$
    Moreover, one can check that $A^T \Omega A = \Omega$, so $A$ is symplectic.

    \end{enumerate}

\end{document}